\documentclass{article}
\usepackage[utf8]{inputenc}
\usepackage{fullpage}
\usepackage{hyperref,verbatim}
\usepackage{enumitem,bm}
\usepackage{url}
\usepackage{wrapfig}
\usepackage{amsthm,amsmath,amssymb}
\usepackage{xcolor}
\usepackage{pgf,tikz,tkz-graph,subcaption}
\usepackage{tkz-berge}
\usetikzlibrary{arrows,shapes}
\usetikzlibrary{decorations.pathreplacing}
\usetikzlibrary{calc}
\usepackage{xfp}
\usepackage{algorithm}
\usepackage{environ}

\newtheorem{defi}{Definition}
\newtheorem{conj}[defi]{Conjecture}
\newtheorem{cor}[defi]{Corollary}
\newtheorem{thr}[defi]{Theorem}
\newtheorem{lem}[defi]{Lemma}
\newtheorem{obs}[defi]{Observation}
\newtheorem{prop}[defi]{Proposition}
\newtheorem{exam}[defi]{Example}

\newtheorem{claim}[defi]{Claim}
\newcommand*{\myproofname}{Proof}
\newenvironment{claimproof}[1][\myproofname]{\begin{proof}[#1]}{\end{proof}}

\newcommand*{\floorfrac}[2]{\mathopen{}\left\lfloor\frac{#1}{#2}\right\rfloor\mathclose{}}

\DeclareMathOperator{\lcs}{lcs}
\DeclareMathOperator{\scs}{scs}
\DeclareMathOperator{\diam}{diam}
\DeclareMathOperator{\sn}{sn}

\DeclareMathOperator{\ivs}{ivs}
\DeclareMathOperator{\SUD}{SUD}

\newcounter{row}
\newcounter{col}

\newcommand\setrow[4]{
  \setcounter{col}{1}
  \foreach \n in {#1, #2, #3, #4} {
    \edef\x{\value{col} - 0.5}
    \edef\y{4.5 - \value{row}}
    \node[anchor=center] at (\x, \y) {\n};
    \stepcounter{col}
  }
  \stepcounter{row}
}

\newcommand\setrowup[9]{
  \setcounter{col}{1}
  \foreach \n in {#1, #2, #3, #4, #5, #6, #7, #8, #9} {
    \edef\x{\value{col} - 0.5}
    \edef\y{9.5 - \value{row}}
    \node[anchor=center] at (\x, \y) {\n};
    \stepcounter{col}
  }
  \stepcounter{row}
}

\title{Extremal and monotone behaviour of the Sudoku number and related critical set parameters}
\author{
Stijn Cambie\thanks{Extremal Combinatorics and Probability Group (ECOPRO), Institute for Basic Science (IBS), Daejeon, South Korea, supported by the Institute for Basic Science (IBS-R029-C4),
E-mail: {\tt stijn.cambie@hotmail.com} or {\tt stijncambie@ibs.re.kr}.}}

\date{}

\begin{document}

\maketitle
\begin{abstract}
    The Sudoku number has been defined under various names, indicating it is a natural concept. There are four variants of this parameter, that can be related to the maximum and minimum size of a critical set in a graph colouring problem.
    For each of these four related parameters, we present some simple characterizations of the graphs attaining the maximum possible values.
    As a main result, we answer a question by Cooper and Kirkpatrick, showing that there is monotone behaviour in the number of colours for only two of the four parameters.
    We investigate the monotone behaviour for the subgraph-order as well. For Latin squares and the Sudoku, we solve some variants for hypergraph colouring.
\end{abstract}

\section{Introduction}

Questions about extendability of partial colourings have been posed long ago, see e.g. the survey of Tuza~\cite{Zsolt97survey}. Here one is mainly interested in the possibility of at least one extending colouring, e.g. under the restriction that the precoloured vertices are at some of distance of each other, the latter being studied e.g. for the planar case in~\cite{Albertson98}.
We will consider cases in which the colouring can be extended in a unique way, and for this we state the definitions of a determining - and critical set.

\begin{defi}
    A \emph{determining set} of vertices in a graph $G = (V, E)$ with respect to a proper vertex colouring $c \colon V \to [k]$ is a set $S \subset V$ for which there is no proper $k$-colouring $c'$ of $G$ different from $c$ for which $c'(s)=c(s)$ for every $s \in S.$
\end{defi}

\begin{defi}
    A \emph{critical set} $S$ of vertices in a graph $G = (V, E)$ with respect to a proper vertex colouring $c : V \to [k]$ is a 
minimal determining set $S$ for the pair $(G, c)$, i.e. for every $S' \subset S$ with $S'\not=S$, there does exists a proper $k$-colouring $c'$ different from $c$ for which $c'(s)=c(s)$ for every $s \in S'.$
\end{defi}

\begin{defi}
    Given a proper $k$-colouring $c$ of $G$, a vertex $v$ is called \emph{fixed} for $(G,c)$ if $c(N[v])=[k],$ i.e. the neighbours of $v$ have been assigned all colours from $[k]$ different from $c(v).$
\end{defi}

The following basic observation will be used a few times.

\begin{obs}\label{obs:1}
    A set of fixed vertices for $(G,c)$ which are all coloured with the same colour form an independent set and for any independent set of fixed vertices, its complement is a determining set for $(G,c)$.
\end{obs}

Recently, Lau et al.~\cite{Lau22} used the notion $\sn(G)$ for the smallest possible critical set among all $\chi$-colourings of a graph. 
This term, called the Sudoku number, is inspired by the minimum clue number of a Sudoku.
The Sudoku number has been defined before as smallest defining set (see e.g.~\cite{M98} and references therein), smallest critical set~\cite{CK14} or forcing chromatic number~\cite{HSV07}. A similar notion is the forced matching number of graphs, see~\cite{AHM04}.
For a survey on defining/ critical sets, we refer to~\cite{DMRP03}.
The fact that it reappears in literature under various names and inspired by different applications or interests, indicates that it is a natural parameter.

In this paper, we extend the notion $\sn(G)$ to $k$-colourings, where $k \ge \chi(G).$ This will immediately address~\cite[Prob.~4.1]{Lau22}.
We do the same for three related notions defined by Cooper and Kirkpatrick~\cite{CK14}.

\begin{defi}
    The Sudoku number $\sn(G,k)$ is the minimum number of vertices in $G$ that have to be coloured in a partial colouring $c'$, such that there exists a unique $k$-colouring $c$ extending $c'$.
    When $k=\chi(G),$ we denote $\sn(G,k)$ also by $\sn(G).$
\end{defi}

\begin{defi}
    For a fixed $k$-colouring $c$ of $G$, 
    $\scs(G,c)$ and $\lcs(G,c)$ are the size of a smallest resp. largest critical set for the pair $(G,c).$
    The values $\underline{\lcs}(G,k), \overline{\scs}(G,k)$ and $\overline{\lcs}(G,k)$ are equal to respectively the minimum of $\lcs(G,c)$, maximum of $\scs(G,c)$ and maximum of $\lcs(G,c)$
    over all $k$-colourings $c$ of $G$.
\end{defi}

It is now immediate to see that if $G$ is the disjoint union of graphs $G_1, G_2, \ldots, G_m$, then
$\sn(G,k)= \sum_{i=1}^m \sn(G_i, k)$ and the same does hold for $\underline{\lcs}, \overline{\scs}$ and $\overline{\lcs}.$ 
As such, in the remaining of this paper, we are only dealing with connected graphs.
We will prefer to use $\sn(G)$ over 
$\underline{\scs}(G)$, as this highlights that it is the most natural and most studied parameter of the four numbers associated with critical sets. 

Considering $k$-colourings with $k>\chi(G)$ has been done in~\cite{D05} and was suggested in Cooper and Kirkpatrick~\cite{CK14} as well, but most research done so far was for the case $k=\chi.$
In particular, the Sudoku number of certain graphs have been determined in the past. Denoting the cartesian product of two graphs with $\square$, 
\begin{itemize}
    \item uniquely colourable graphs $G$ satisfy $\sn(G)=\chi(G)-1$ and hence $\sn(G)=1$ if $G$ is bipartite,
    \item the Petersen graph $G$ satisfies $\sn(G)=4$, 
    \item $\sn(C_{2n+1})=n+1$
and $\sn(K_2 \square C_{2n+1})=n+1$,
\item $\sn(K_n \square K_m)=m(n-m)$ if $n \ge m^2$,
\item $\sn(C_m \square K_n)=m(n-3)$ and $\sn(P_m \square K_n)=m(n-3)+2$ for $n \ge 6$,
\end{itemize}
and some more graphs such as $C_m \square K_n, P_m \square K_n, W_n, C_{2n}(K_m)$ and $C_{2n}(K^-_m).$ See~\cite{DMRP03,Lau22}.

For the traditional Sudoku graph $G$, which is the Latin square $K_9\square K_9$ with the edge set of $9$ additional $K_9$s added, it is known (among those who believe the computer-aided proof) that $\sn(G)=17$, by~\cite{MTC14}.
Latin squares are one of the leading cases and have been investigated extensively, see~\cite{HY18} and the references therein. Here it is shown that $\sn(K_n\square K_n)=\Theta(n^2)$, while it has been conjectured before that $\sn(K_n\square K_n)=\floorfrac{n^2}{4}.$

As a small example of our extension, we give the suduko numbers of the $3 \times 3$ Latin square.
\begin{exam}
$\sn(K_3 \square K_3,k)= \begin{cases}
2 & \text{ if } k=3,\\
5& \text{ if } k=4,\\
7 & \text{ if } k=5,\\
9& \text{ if } k\ge 6.\\
\end{cases}$
\end{exam}

Determining $\sn(K_n\square K_n,k)$ for most of the values $n\le k \le 2n-1$, as well as for the related Suduko versions (when $n$ is a square), would be a challenge, as it is for the case where $k=n$.

The fact that $\sn(K_n \square K_n,k)=n^2$ whenever $k \ge 2n$ is not surprising. It is a special case of the two basic extremal results proven in Section~\ref{sec:2basisresults}. 

\begin{thr}\label{thr:2BasicExtrRes}
    Let $G$ be a graph of order $n$ and $k\ge \chi(G)$ be a number.
    Then $\sn(G,k)=n$ if and only if $k>\Delta+1,$
    and $\sn(G) =n-1$ if and only if $G = K_n$.
\end{thr}

The latter result solves~\cite[Conj.~4.1]{Lau22}, which was also proved by Pokrovskiy~\cite{P22}\footnote{The two independent proofs were communicated to Lau et al, on the same day.}.
We also consider these two basic extremal results for the other $3$ related parameters.
In Section~\ref{sec:n-2} we investigate the $\sn(G) =n-2$ case.

One would expect that $\sn(K_n\square K_n,k)$ is non-decreasing in terms of $k$.
In Section~\ref{sec:monotonedecreasing_colours}, we observe that $\overline{\scs}(G,k)$ and $\overline{\lcs}(G,k)$ are monotone non-decreasing in $k$ for every graph $G$, while this is not the case for $\sn(G,k)$ and $\underline{\lcs}(G,k).$ This answers a question of~\cite{CK14}.
We also consider this question about the monotone behaviour of the parameters for the subgraph-relation, under the (necessary) condition of fixed chromatic number. This is done in Section~\ref{sec:monotonedecreasing_subgraph}.
Finally, in Section~\ref{sec:hypergraph} we consider hypergraph-versions and determine the parameters for the Suduko and Latin squares (which are hard and interesting cases in the graph case).

\newpage

\section{Two basic extremal results for the maximum possible Sudoku number}\label{sec:2basisresults}

In this section, we prove the two basic extremal results of Theorem~\ref{thr:2BasicExtrRes} and extend with some additional characterizations for $\overline{\scs}, \underline{\lcs}$ and $\overline{\lcs}$.

\subsection{Large Sudoku numbers when $k=\chi$}

\begin{thr}\label{thr:snG=n-1}
    For a graph $G$ of order $n$, $\sn(G) = n - 1$ if and only if $G = K_n$.
    Also $\overline{\scs}(G)=n-1$ if and only if $G = K_n$.
\end{thr}

\begin{proof}
    One direction is easy. Since $K_n$ is a uniquely colourable graph, we know  $\sn(K_n)=\overline{\scs}(K_n)=n-1.$
    So now assume $G$ is not equal to $K_n$.
    Take an arbitrary $k$-colouring $c$ of $G$, with $k=\chi(G)$.
    Since no $(k-1)$-colouring does exist, for every colour $i$, there is a fixed vertex coloured $i$.
    Select a critical vertex $v_i$ for every colour $i$.
    If $\{v_1,\ldots,v_k\}$ forms a clique $K$, since $G$ itself is not a clique, there is a vertex, without loss of generality $v_1$, with a neighbour $u$ outside $K$.
    Let $j$ ($\not=1$) be the colour associated with $u$.
    If we uncolour $v_1$ and $v_j$, then $v_1$ can only be coloured with $1$ and next $v_j$ only with $j$, so $\scs(G,c) \le n-2.$

    If $\{v_1,\ldots,v_k\}$ does not form a clique, e.g. $v_1$ and $v_2$ are not connected, then this $k$-colouring restricted to $G\backslash\{v_1,v_2\}$ is uniquely extendable and so again $\scs(G,c) \le n-2$.
    Since this is true for every $k$-colouring $c$, we conclude that $\sn(G)=\underline{scs}(G) \le \overline{\scs}(G)<n-1$. This proves uniqueness of the extremal graph.
\end{proof}

The characterization of the graphs $G$ for which $\overline{\lcs}(G) = n - 1$ or $\underline{\lcs}(G) = n - 1$ is different.
In the first case, e.g. all odd cycles satisfy this equality. In the second case both $C_5$ and the wheel graph $W_6$, as well as $K_k$ (for $k\ge 3$) with a path connected to one of its vertices satisfy $\underline{\lcs}(G) = n - 1$.

\begin{thr}\label{thr:olcsG=n-1}
    For a graph $G$ of order $n$, $\overline{\lcs}(G) = n - 1$ if and only if $G$ has a vertex $v$ for which $\chi(G \backslash v)=\chi(G)-1=\deg(v).$
\end{thr}

\begin{proof}
    Denote $k=\chi(G).$
    If $G$ has a vertex $v$ as in the statement, we can colour $G \backslash v$ with $k-1$ colours and $v$ needs to get the additional colour. Let $c$ be such a colouring.
    Note that $v$ its neighbours are coloured with the $k-1$ different colours (otherwise $\chi(G)=k-1$).
    We claim that $V \backslash v$ is a critical set for $(G,c).$
    Uncolouring a neighbour $u$ of $v$ as well, would imply that we can switch the colour of $u$ and $v$.
    Uncolouring a nonneighbour $u$ of $v$, would imply that $u$ can be coloured with either $c(u)$ or $k$. 
    So indeed $V \backslash v$ is a critical set for $c$.
    
    In the other direction, assume $\overline{\lcs}(G) = n - 1$.
    This implies that there exists a colouring $c$ of $G$ and a vertex $v$ for which $V\backslash v$ is a critical set for $(G,c).$
    In particular the neighbours of $v$ received the $k-1$ colours different from $c(v)$.
    Without loss of generality, we have $c(v)=k.$
    For every colour $i \in [k-1]$, there is at least one fixed vertex $u \in V$ for which $c(u)=i$ which could not be recoloured, i.e. for which the colourneighbourhood is $c(N(u))=[k]\backslash i.$
    Since $V\backslash v$ is a critical set, each such $u$ is a neighbour of $v$.
    If $\deg(v)>k-1$, there would be a colour $i$ appearing twice in $N(v)$. Uncolouring a fixed (for $(G,c)$) neighbour $u$ of $v$ which has been assigned that colour $i$, would lead to a uniquely extendable precolouring, since $c(v)$ is still known and thereafter $c(u)$ is known.
    Thus $\deg(v)=k-1.$
    Since every fixed vertex belongs to $N[v]$ (since $V\backslash v$ is a critical set), all vertices different from $v$ can be recoloured with a colour in $[k-1]$ if necessary and thus $\chi(G \backslash v)=k-1.$
\end{proof}

A similar characterization of all graphs for which $\underline{\lcs}(G) = n - 1$ seems to be much harder. We give two examples of families of graphs satisfying this equality. 

\begin{exam}
    Let $G=K_n \backslash C_5$, where $n \ge 5.$
    Then $\omega(G)=n-3<n-2=\chi(G)$ and $\underline{\lcs}(G) = n - 1$.
\end{exam}

\begin{proof}
    Let $c$ be any $(n-2)-$colouring of $G$.
    The $n-5$ vertices with degree $n-1$ all need to be coloured with a different colour.
    The other $5$ vertices are coloured with $3$ colours and one of its vertices, $v$, is coloured with a unique colour.
    Now $V \backslash v$ is a critical set for $(G,c),$ since for every other vertex $v'$ we would also have a proper colouring if we switch the colours $c(v)$ and $c(v').$
\end{proof} 

\begin{exam}
    Let $G=(V,E)$ be a graph for which there is a subset $U \subset V$ such that $G[U]=K_k$, $\lvert N(U) \rvert + \# \{ u \in U \mid \deg(u)>k-1 \} <k$ and $\deg(v)<k-1$ for every $v \in V \backslash U.$
    Then $\underline{\lcs}(G) = n - 1.$
\end{exam}

\begin{proof}
    Since the degeneracy of $G$ equals $k-1$ and the clique number $k$, $\chi(G)=k.$
    Let $c$ be an arbitrary proper $k$-colouring of $G$.
    The condition $\lvert N(U) \rvert < \# \{ u \in U \mid \deg(u)=k-1 \} $ implies that there is some vertex $u \in U$ for which $ \deg(u)=k-1 $ and $c(u)$ does not appear as a colour on $N(U).$
    We claim that $V\backslash u$ is a critical set for $(G,c).$
    It is a determining set since the $K_k$ gets $k$ different colours.
    Every critical set contains $V\backslash U$, since the colour of these vertices cannot be determined as they have degree less than $k-1.$
    Finally, for every $u'\in U$ with $u \not= u'$ we can switch $c(u)$ and $c(u')$ to get an other proper colouring by the choice of $u$.
    Hence $V\backslash \{u,u'\}$ is not a determining set.
    As $c$ was taken arbitrary, we conclude that $\underline{\lcs}(G) = n - 1.$
\end{proof}

Even while this is an infinite family, it may be that there are not so many graphs for which $\omega(G)<\chi(G)$ do satisfy $\underline{\lcs}(G) = n - 1$.

\subsection{Largest Sudoku numbers for $k>\chi$}

When $k>\chi(G),$ it is possible that there are no fixed vertices for a pair $(G,c)$ and the critical set attains the maximum order $n$.
In this subsection, we state the (necessary and sufficient) conditions for $\sn(G,k)=n$ and the three analogues.

\begin{thr}\label{thr:snGk=n}
    Let $G$ be a graph and $k\ge \chi(G)$ be a number.
    Then $\sn(G,k)=n$ if and only if $k>\Delta+1.$
    Also $\underline{\lcs}(G,k)=n$ if and only if $k>\Delta+1.$
\end{thr}

\begin{proof}
    One direction is trivial. If $k>\Delta+1$, then one can greedily extend any $k$-colouring, with at least two options in every step of assigning a colour to an uncoloured vertex. This implies that the only possible critical set is $V$.
    
    Next, we prove the other direction. We do this first for the case where $k=\Delta+1.$ This is an easy case, since one can pick one vertex $u$ for which $\deg(u)=\Delta(G)$ and colour its neighbours with $\Delta$ different colours. This colouring can be greedily extended to a $k$-colouring of $G \backslash u.$ By construction, there is only one colour available for $u$ and as such $\sn(G,k)<n.$

    Now assume $G$ is a minimal counterexample to the statement, i.e. $\chi(G) \le k \le \Delta(G)$ and $\sn(G,k)=n.$
    Let $v$ be a vertex whose degree is strictly smaller than $k$, if there is such a vertex. Since $k \le \Delta(G \backslash v)+1,$ we know that (at least one component of) $G \backslash v$ is not a counterexample and as such there is precolouring with at most $n-2$ vertices coloured, which has a unique extension.
    Since $\deg(v)<k,$ we can assign a colour to $v$ (if that colour is unique, we do not have to do) such that the precolouring on $G$ can be extended in only one way.
    Finally, if there is no vertex $v$ with $\deg(v)<k$, then one can take a $\chi$-colouring of $G$.
    Since there is no colouring with $\chi-1$ colours, there is a vertex $v$ such that its neighbourhood contains all $\chi$ colours different from its own colour.
    Since $\deg(v) \ge k$, we can recolour some of its neighbours with the $k-\chi$ unused colours, in such a way that that every additional colour appears precisely once and $N(v)$ is coloured with $k-1$ different colours.
    If we now uncolour $v$, then this can be extended only in one way, i.e. $V \backslash v$ is a determining set. So $G$ was not a counterexample and we conclude.
\end{proof}

Note that as a corollary of Theorem~\ref{thr:snG=n-1} and~\ref{thr:snGk=n}, we also know the characterization of all graphs (not necessarily connected), for which $\sn(G)=n-1.$ It is the union of a complete graph $K_k$ and some graphs with maximum degree at most $k-2.$

We also observe that the characterization is different when the parameter is $\overline{\scs}$ or $\overline{\lcs}.$

\begin{thr}\label{thr:overline=n}
    Let $G$ be a graph and $k\ge \chi(G)$ be a number.
    Then $\overline{\scs}(G,k)=n$ (idem dito for $\overline{\lcs}(G,k)=n$) does hold if and only if $k>\chi.$
\end{thr}

\begin{proof}
    It was already known (e.g. by \cite[Prop.~4]{CK14}) that if $k=\chi$, $\overline{\scs}(G,k) \le \overline{\lcs}(G,k)\le n-1$. The latter is true since for every $\chi$-colouring and colour $i \le \chi$, there is a vertex with colour $i$ which cannot be recoloured (as otherwise $\chi(G) \le \chi-1$, contradiction).
    Now if $k>\chi$, take any $\chi$-colouring.
    In that case the only critical set is the whole set.
\end{proof}

\section{A strengthening of the extremal result for the Sudoku number}\label{sec:n-2}

In Theorem~\ref{thr:snG=n-1}, we gave a short proof for an elementary characterization of the graphs with largest Sudoku number, being precisely the cliques.
Except from odd cycles, Brook's theorem states that cliques are the only graphs attaining the maximum possible chromatic number in terms of the maximum degree. Reed~\cite{Reed99} considered the up to one extremal case where $\chi=\Delta$ and concluded that for large $\Delta$, $K_{\Delta}$ has to be a subgraph of a graph satisfying that equality.
We conjecture, that except from $2$ sporadic counterexamples, the same behaviour is true when considering the equality $\sn(G)=n(G)-2.$

\begin{conj}\label{conj:sn=n-2_no_Kn}
    The cycle $C_5$ and the Moser-spindle (drawn in Figure~\ref{fig:moserspindle}) are the only $2$ graphs for which $\sn(G)=n(G)-2$ which are not perfect, i.e. $\omega(G)<\chi(G).$
\end{conj}

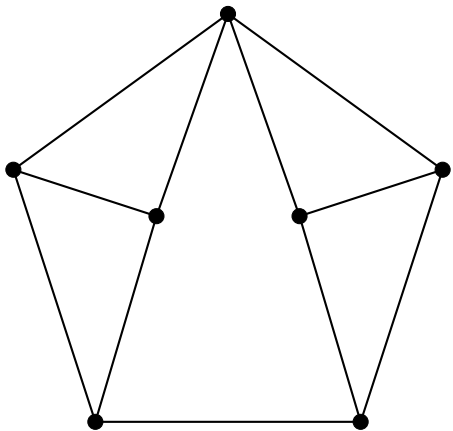
\begin{figure}[h]
\begin{center}

    \begin{tikzpicture}
    {
    \foreach \x in {0,1,2,3,4}{
	\draw[thick] (\x*72-54:3) -- (\x*72+18:3);
	\draw[fill] (\x*72-54:3) circle (0.1);
	}
	
	\foreach \x in {0,1,2}{
	\draw[thick] (\x*72-54:3) -- (18:1);
	\draw[thick] (\x*72+90:3) -- (216-54:1);
	}
	
	\draw[fill] (72-54:1) circle (0.1);
    \draw[fill] (216-54:1) circle (0.1);

	}
	\end{tikzpicture}
\end{center}
\caption{The Moser-spindle}\label{fig:moserspindle}
\end{figure}

Theorem~\ref{thr:snG=n-1} did also hold for $\overline{\scs}$, but we observe that there are infinitely many graphs for which $\omega(G)<\chi(G)$ satisfying $\overline{\scs}(G)=n-2.$ As such, we will focus on the $\sn$ case in the remaining of this section.

\begin{prop}
    There are infinitely many graphs $G$ with $\omega(G)<\chi(G)$ for which $\overline{\scs}(G)=n-2.$
\end{prop}

\begin{proof}
Let $C_5$ be a cycle with vertices $v,w,x,y$ and $z$, in this order.
    Let $G=(V,E)$ be this cycle, where $w$ and $y$ are blown-up with a $K_p$, and $x$ and $z$ are blown-up with a $K_q.$
    Then clearly $\chi(G)=p+q+1$ (observe that the independence number $\alpha(G)=2$).
    Let $c$ be the colouring of $G$ for which the vertices in the two $K_p$s are coloured with $[p]$, the vertices in the $K_q$s with $[p+1..p+q]$ and the remaining vertex $v$ with $p+q+1.$
    Any critical set has to contain all vertices in the cliques corresponding with $x$ and $y$, as these vertices can be coloured with both its current colour and the colour $p+q+1.$
    It is not hard to see that $V \backslash v$ is a critical set, since if some vertex in the cliques corresponding with $w$ or $z$ would not belong to it as well, one can switch the colours of $v$ and that vertex.
    Since a critical set $S$ cannot miss more than one vertex of a clique (otherwise colours can be switched), we find that $\lvert S \rvert \ge n-2.$ By Theorem~\ref{thr:snG=n-1}, this implies that $\overline{\scs}(G)=n-2.$
\end{proof}

In the case of Reed's strengthening of Brook's theorem, when $G$ is a graph for which $K_{\Delta} \subset G$ we also know that $\chi(G) \ge \Delta$ and the characterization is done.
This is not the case in our setting, 

The following list contains examples of graphs $G$ for which $\omega(G)=\chi(G)=k$ and $\sn(G)=n-2$ :
    \begin{itemize}
        \item $G$ contains a clique $K_k$ such that all vertices in the clique have degree at most $k$ and the other vertices have degree at most $k-2$
        \item $G$ contains a clique $K_k$, which is connected by a bridge (cut-edge) to $G\backslash K_k$, where all vertices of $V(G) \backslash V( K_k)$ have degree at most $k-2$, except the endvertex of the bridge which has degree $k-1$
        \item $G$ contains a clique $K_k$, all the vertices of this clique except from one have degree $k-1$ and all vertices in  $V(G) \backslash V( K_k)$ have degree at most $k-2$
        \item $G$ is the union of a clique $K_k$ and a complete bipartite graph $K_{a,b}$, where all vertices in the partition class of size $a$ are connected with $c$ vertices of the clique $K_k$, where $2\le c$ and $c+b\le k-2.$
        Here $a$ can be arbitrary large.
    \end{itemize}

Note that the last example shows that $k \le \Delta$ does not imply that there is a vertex with maximum degree which can be fixed for some colouring.

In the remaining of this section, we will give some directions towards Conjecture~\ref{conj:sn=n-2_no_Kn}.
Our first lemma, will use the definition of independent chromatic vertex stability number (defined in~\cite{ABKM22}). 

\begin{defi}
     The size of the smallest independent set $S$ such that $\chi(G \backslash S)<\chi(G)$ is the \textit{independent chromatic vertex stability number}, $\ivs_\chi(G)$.
\end{defi}

It is important to note that $\ivs_{\chi}(G)=1$ does not imply that the graph is colour-critical, as it implies that there is one (but not all) vertex $v$ for which $\chi(G \backslash v)<\chi(G).$

\begin{lem}
    If a graph $G$ satisfies $\sn(G)=n-2$, then $\ivs_{\chi}(G)=1.$
\end{lem}

\begin{proof}
   Let $k=\chi(G).$ Observe that we can assume $k \ge 3$ since the cases $k \in \{1,2\}$ are trivial (remember that $G$ is connected and $\sn(G)=1$ when $G$ is bipartite).
   If there is a $k$-colouring $c$ such that there is some colour $i$ for which there are at least $3$ fixed vertices coloured with $i$ , we would have $\sn(G)< n-2$ (by Observation~\ref{obs:1}).
   This immediately implies that $\ivs_{\chi}(G)\le 2.$
   If $\ivs_{\chi}(G)=2$ and $\sn(G)=n-2$, then for every $k$-colouring and every $i \in [k]$, there are exactly $2$ fixed vertices of every colour.
   Let $1 \le i ,j \le k$ be two different colours. If there would be a fixed vertex coloured $j \not=i$ that is connected to none of the two fixed vertices coloured $i$, the complement of these three vertices would be a critical set, in which case $\sn(G)< n-2$.
   Hence for every $1 \le i<j \le k,$ $G$ contains a matching between the two fixed vertices coloured $i$ and the two fixed vertices coloured $j$.
   The union of these matchings (restricted to the $2k$ fixed vertices) either gives two cliques $K_k$, or has independence number at least $3.$
   In the first case, since $G$ is connected, there is at least one vertex $v$ in a $K_k$ with an additional neighbour $u$.
   Uncolouring the vertex in that $K_k$ with the same colour as $u$, the vertex $v$, and a vertex in the other $K_k$ (different from $u$), the colouring was uniquely extendable and thus $\sn(G)\le n-3.$
   If the independence number was at least $3$, the independent set corresponds to a set of vertices for which its complement is a critical set (by Observation~\ref{obs:1}), from which $\sn(G)\le n-3$ would follow again.
\end{proof}

\begin{prop}
    The cycle $C_5$ is the only $K_3$-free graph for which $\chi(G)=3$ and $\sn(G)=n-2.$
\end{prop}

\begin{proof}
    Since $\chi(G)=3$ and $\ivs_{\chi}=1,$ there is a vertex $v$ (possible multiple) that belong to all odd cycles.
    Consider the odd-girth of the graph $g$ and let $C_g$ be a cycle (it is induced) of girth $g$.
    If $g$ is at least $7$, one can colour $v$ with $3$ and $G \backslash v$ with $\{1,2\}.$
    Now take a vertex $u$ on $C_g$ at distance at least $3$ form $v$ and colour it in $3$ as well.
    For this colouring, the two neighbours of $v$ on $C_g$ and the neighbour(s) of $u$ on $C_g$ which are at distance at least $3$ from $v$ are all fixed. So (by Observation~\ref{obs:1} again) $\sn(G) \le n-3.$
    
    Next, assume $G$ contains one odd cycle $C_5$, with vertices $u,v,w,x,y$ in this order, and is not $C_5$ itself. Then at least one vertex, without loss of generality $u$, has degree at least $3$ and thus an other neighbour $z$.
    Let $c$ be a $2$-colouring of $G\backslash v$ and (re)colour $y$ and $v$ with $3.$
    Then $V \backslash \{u,v,x\}$ is a determining set for $(G,c)$ and as such $\sn(G) \le n-3.$

    If $G$ has odd-girth $5$ and contains multiple cycles, there is a vertex $v$ (belonging to all odd cycles, i.e. $\chi(G \backslash v)=2$) with at least $3$ neighbours belonging to some odd cycle.
    Let $v$ be a vertex belonging to all odd cycles. If only two of its neighbours belong to some odd cycles, they belong to all odd cycles as well and we can repeat the search (and we have to repeat at most twice).
    Having found such a vertex, we colour $G \backslash v$ with $[2]$ and colour $v$ with $3$. Then those (at least) $3$ neighbours belonging to odd cycles, form a fixed set for the colouring and we conclude again, as the complement is a determining set with at most $n-3$ vertices.
\end{proof}

\begin{lem}
    It is sufficient to consider graphs for which $\omega(G)<\chi(G)=k\ge 4$ and $\delta(G) \ge k-1$ and prove that the Moser-spindle is the only one with $\sn(G)=n-2.$
\end{lem}

\begin{proof}
    We only have to focus on graphs with $\chi(G) \ge 4$ since the case $\chi=3$ has been proven and $\chi=2$ is trivial.
    Let $v$ be a vertex for which $\deg(v) \le k-2$, then $\sn(G \backslash v)=n'-2$ as well, since if not, the colouring $c$ of $G \backslash v$ (observe that $\chi(G \backslash v)=k$) can be extended to $G$ and the critical set for $(G \backslash v,c)$ of size at most $n'-2$ as well to a determining set of size at most $n-2.$ 
    So the only situation in which the statement of the lemma is false, is if there would exist a graph $G$ and vertex $v$ for which $\deg(v) \in \{1,2\}$,
    $ G \backslash v$ is the Moser-spindle and $\sn(G)=n-2.$
    This cannot be the case.
    It has been presented for the possible cases where $\deg(v)=1$ in Figure~\ref{fig:extensionsMoserspindle}, with a determining (critical) set of size at most $n-3$ being presented in black. The case $\deg(v)=2$ can be done by considering $7$ cases, $6$ of them are obtained by adding one of the dotted lines.
    The final case is true as well, and left as a little puzzle for the reader.
\end{proof}

\begin{figure}[h]
\begin{minipage}[b]{.31\linewidth}
\begin{center}
\scalebox{0.8}{
    \begin{tikzpicture}
    {
    \foreach \x in {0,1,2,3,4}{
	\draw[thick] (\x*72-54:3) -- (\x*72+18:3);
	\draw[fill] (\x*72-54:3) circle (0.1);
	}
	
	\foreach \x in {0,1,2}{
	\draw[thick] (\x*72-54:3) -- (18:1);
	\draw[thick] (\x*72+90:3) -- (216-54:1);
	}
	\draw[thick] (0,-1) -- (18-144:3);
	\draw[dotted] (0,-1) -- (18-72:3);
	\draw[fill] (72-54:1) circle (0.1);
    \draw[fill] (216-54:1) circle (0.1);
    \draw[fill] (0,-1) circle (0.1);
   	\coordinate [label=center: \textcolor{red}{4}] (A) at (144-54:3.3);
	 \coordinate [label=center: \textcolor{red}{2}] (A) at (-54:3.3);
     \coordinate [label=center: 3] (A) at (18:3.3);
     \coordinate [label=center: 3] (A) at (144+18:3.3);
    \coordinate [label=center: 1] (A) at (18:0.7);

\coordinate [label=center: 1] (A) at (0,-0.7);
\coordinate [label=center: \textcolor{red}{4}] (A) at (-72-54:3.3);
\coordinate [label=center: 2] (A) at (-144-54:0.7);
}

	\end{tikzpicture}}
\end{center}
\end{minipage}\quad
\begin{minipage}[b]{.31\linewidth}
\begin{center}\scalebox{0.8}{
    \begin{tikzpicture}
    {
    \foreach \x in {0,1,2,3,4}{
	\draw[thick] (\x*72-54:3) -- (\x*72+18:3);
	\draw[fill] (\x*72-54:3) circle (0.1);
	}
	
	\foreach \x in {0,1,2}{
	\draw[thick] (\x*72-54:3) -- (18:1);
	\draw[thick] (\x*72+90:3) -- (216-54:1);
	}
	\draw[thick] (72-54:3) -- (2.8,3);
	
	\draw[fill] (72-54:1) circle (0.1);
    \draw[fill] (216-54:1) circle (0.1);
    
   \draw[fill] (2.8,3) circle (0.1);
    
    \draw[dotted] (144-54:3) -- (2.8,3);
	\draw[dotted] (216-54:3) -- (2.8,3);
	\draw[dotted] (288-54:3) -- (2.8,3);
	\draw[dotted] (-54:3) -- (2.8,3);
	
	\coordinate [label=center: 4] (A) at (144-54:3.3);
	 \coordinate [label=center: 3] (A) at (-54:3.3);
     \coordinate [label=center: \textcolor{red}{2}] (A) at (18:3.3);
     \coordinate [label=center: \textcolor{red}{3}] (A) at (144+18:3.3);
    \coordinate [label=center: \textcolor{red}{1}] (A) at (18:0.7);

\coordinate [label=center: 1] (A) at (3.1,2.5);
\coordinate [label=center: 2] (A) at (-72-54:3.3);
\coordinate [label=center: 1] (A) at (-144-54:0.7);
	}
	\end{tikzpicture}}
\end{center}
\end{minipage}\quad\begin{minipage}[b]{.31\linewidth}
\begin{center}
    \scalebox{0.8}{
    \begin{tikzpicture}
    {
    \foreach \x in {0,1,2,3,4}{
	\draw[thick] (\x*72-54:3) -- (\x*72+18:3);
	\draw[fill] (\x*72-54:3) circle (0.1);
	}
	
	\foreach \x in {0,1,2}{
	\draw[thick] (\x*72-54:3) -- (18:1);
	\draw[thick] (\x*72+90:3) -- (216-54:1);
	}
	\draw[thick] (90:3) -- (0,0);
	\draw[dotted] (90+144:3) -- (0,0);
	\draw[fill] (72-54:1) circle (0.1);
    \draw[fill] (216-54:1) circle (0.1);
    
   \draw[fill] (0,0) circle (0.1);
    \coordinate [label=center: \textcolor{red}{4}] (A) at (144-54:3.3);
	 \coordinate [label=center: 2] (A) at (-54:3.3);
     \coordinate [label=center: 3] (A) at (18:3.3);
     \coordinate [label=center: 3] (A) at (144+18:3.3);
    \coordinate [label=center: \textcolor{red}{1}] (A) at (18:0.7);

\coordinate [label=center: 2] (A) at (0.3,0);
\coordinate [label=center: \textcolor{red}{4}] (A) at (-72-54:3.3);
\coordinate [label=center: 1] (A) at (-144-54:0.7);
	}
	\end{tikzpicture}}
\end{center}
\end{minipage}
\caption{Examples of critical sets for extensions of the Moser-spindle}\label{fig:extensionsMoserspindle}
\end{figure}
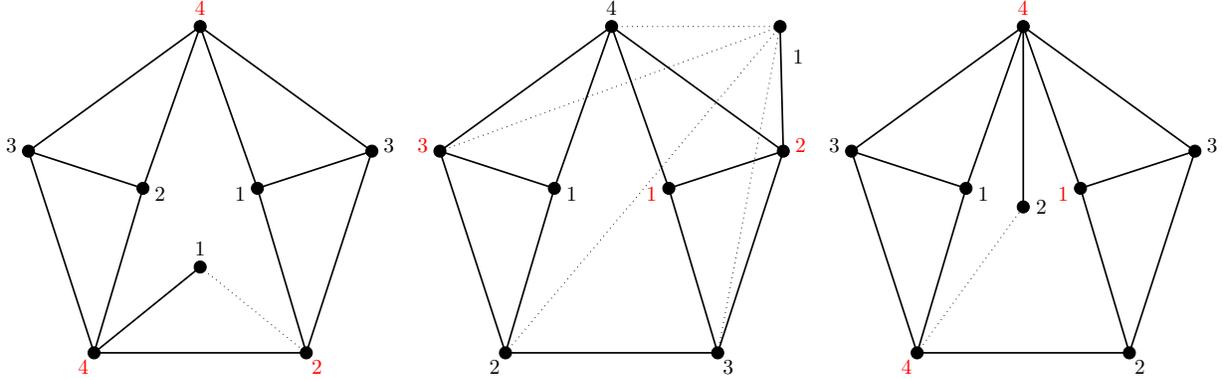

\section{Monotone non-decreasing parameters in terms of number of colours}\label{sec:monotonedecreasing_colours}

In this section, we address~\cite[Prob.~3]{CK14} (third problem in their conclusion section).

In the case of $\sn(G,k),$ we will start exploring the tree and bipartite graph case.
While $\sn(G)=\sn(T)=1$ for bipartite graphs $G$ and thus trees $T$, we note that the case $k>2$ is harder.
For trees, but not for bipartite graphs,
an upper bound in terms of the $r$-domination-number $\gamma_r(G)$ is valid.
The $r$-domination-number of a graph $G$ is the order of a smallest subset $D \subset V$ such that every vertex not in $D$ has at least $r$ neighbours in $D$. This notion was defined in~\cite{FJ85} and is non-decreasing in $r$, i.e. $\gamma_r(G) \le \gamma_{r+1}(G).$

\begin{prop}
    Let $T$ be a tree and $r \ge 1$ an integer. Then $\sn(T, 1+r) \le \gamma_r(T).$
    When $G$ is a bipartite graph and $r \ge 2$, $\sn(G, 1+r)$ cannot be bounded in terms of $\gamma_r(G).$
\end{prop}
\begin{proof}
    Let $D$ be a $r$-dominating set of $T$.
    Since every vertex $v$ is a cut vertex in a tree $T$, one can permute the colour classes in any component of $T \backslash v$ without causing any issues at vertices different from $v$. This implies that one can precolour $D$ such that for every vertex $v$ in $V \backslash D$ its neighbourhood $N(v)$ has been coloured with $r$ different colours and such that for any two neighbours $u,v \in V \backslash D$ the neighbourhoods $N(u)$ and $N(v)$ are missing a different colour.
    From this, we conclude. It is not hard to see that the inequality is not sharp in general. For example a complete binary tree of height $h \ge 2$, satisfies $\sn(T,3)=2^h$, while $\gamma_2(T)=2^h+2^{h-2}+\ldots=\floorfrac{2^{h+2}}{3}$.
    
    Now for every $r\ge2$, we construct a bipartite graph for which $\sn(G, 1+r)>\gamma_r(G).$
    Let one bipartition class $U$ have $r+2$ vertices and the other one, $W$, $3\binom{r+3}{3}$.
    Let the vertices in $W$ have the $\binom{r+2}{r}$ different possible $r$-sets in $\binom{U}{r}$ as their neighbourhoods, where each such neighbourhood appears $r+3$ times.
    Since $U$ is a $r$-dominating set of $G$, we have $\gamma_r(G)\le r+2$ (actually equality does hold).
    On the other hand, a $(r+1)$-colouring of $U$ would have $2$ vertices in the same colour by the pigeon hole principle. Hence at least $r+3$ vertices in $W$ would not be fixed and thus belong to any critical set, i.e. $\sn(G, 1+r)>r+2.$ 
\end{proof}

Next, we prove that there is a polynomial algorithm to determine the Sudoku number $\sn(T,k)$ for any tree $T$ and integer $k \ge 3.$

\begin{prop}\label{propälgT}
    There does exist a polynomial algorithm to compute $\sn(T,k)$ for any tree $T$ and $k \ge 3.$
\end{prop}

\begin{proof}
    If the tree is a star, $\sn(T,k)=n-1$ (if $k \le n$), or $n$ if $k >n.$
    Now assume that $\diam(T)>2.$
    Pick a diameter of $T$ and let $x$ be a neighbour of one of the endpoints of the diameter.
    Let $x$ be adjacent to $\deg(x)-1=\ell>0$ leaves and one non-leaf $z$. We denote the set of these leafs by $L$.
    If $\ell \le k-3$, then all leaves and $x$ itself have to be precoloured and so $\sn(T,k)=\sn(T \backslash L,k) + \ell.$ 
    If $\ell \ge k-1,$ then we can precolour $L$ in such a way that $x$ can only be assigned one possible colour.
    Since $x$ has to be precoloured in $T \backslash L$ (being a leaf), this time we have $\sn(T,k)=\sn(T \backslash L,k) + \ell-1.$ 
    Finally, if $\ell=k-2,$ we prove that $\sn(T,k)=\sn(T \backslash L\backslash x,k) + \ell.$ 
    First observe that $\sn(T,k)\le \sn(T \backslash L\backslash x,k) + \ell,$ since one can precolour $T \backslash L\backslash x$ in such a way that all vertices get a colour at the end, and precolour $L$ in such a way that together with the final colour of $z$ it determines the colour of $x$ uniquely.
    In the other direction, we have $ \sn(T \backslash L\backslash x,k)\le \sn(T,k)- \ell.$
    The smallest set of precoloured vertices in $T$ will contain all of $L$.
    If $x$ is not precoloured, $z$ has to be uniquely determined first.
    If $x$ is precoloured as well, we can colour $z$ instead and do not need to colour both $L$ and $x.$
    This implies that we can easily determine $\sn(T,k)$ once we know $\sn(T',k)$ for some subtree $T'$ and so in at most $n$ steps, we have determined $\sn(T,k)$.
\end{proof}

By a bit of case distinction, one can observe that for a leaf $x$ of tree, $\sn(T \backslash x,k) \le \sn(T,x) \le \sn(T \backslash x,k)+1$.
Hence, as a corollary of the algorithm in Proposition~\ref{propälgT}, we observe that for trees, the parameter $\sn$ is monotone. On the other hand, this is not true in general, even in the class of bipartite graphs.
\begin{cor}
    For trees $T$, the numbers $\sn(T,k)$ are monotone non-decreasing in $k.$
\end{cor}

\begin{thr}
    There are (bipartite) graphs $G$ and integers $k$ for which $\sn(G,k)< \sn(G,k-1),$ i.e. $\sn(G,K)$ is not monotone non-decreasing in $k.$
\end{thr}

\begin{proof}
    Let $G_k=K_{k,k}\backslash M$ with $M$ a perfect matching of $K_{k,k}$. 
    We compute two Sudoku numbers of this graphs in the following claims.
    
    \begin{claim}\label{clm:sn(G,k-1)}
        For $k\ge 3$ the Sudoku number $\sn(G_k,k-1)$ equals $2k-5.$
    \end{claim}
    \begin{claimproof}
        We will prove this by induction.
        For $k=3,$ since $G_3$ is bipartite, we have $\sn(G_3,2)==\sn(G_3)=1= 2 \cdot 3-5.$ So the base case is true.
        
        Now let $k \ge 4.$ 
        Let $c$ be a proper $(k-1)$-colouring of $G_k$.
        Observe that for both sides of the partition, call them left and right, there is a colour appearing at least twice (and hence not at the other side).
        Let $1$ be such a colour at the left and $2$ be a colour appearing multiple times at the right.
        
        First, assume there is no colour appearing on both sides.
        If both sides contain at least $2$ colours, we would be unable to extend a partial colouring uniquely, i.e. $V$ is the only critical set.
        So assume instead that (without loss of generality) the left is completely coloured with $1.$
        In this case, the right side completely belongs to any critical set.
        The right side needs to contain all colours, to ensure that a critical set different from $V$ does exist.
        There are $k-2$ colours appearing on $k$ vertices, hence at least $k-4$ vertices on the right have a colour that appears only once.
        For each of these vertices, the neighbour in $M$ is an element of any critical set.
        So the critical set contains at least $2k-4$ vertices.
        
        Next, assume there is a colour $i$ which appears on both sides of the partition in which case it appears precisely once, at the end-vertices of an edge $e$ of $M$. 
        Then $V(G_k)\backslash V(e)$ is coloured with $k-2$ colours and the constraints are the same as for $G_{k-1}.$
        A critical set clearly needs to contain the vertices coloured with $i$, since otherwise these could be coloured with $1$ or $2.$
        The remaining of the critical set, has to be a critical set of the $(k-2)$-colouring of $G_{k-1}.$
        As such, this critical set contains at least $2+2(k-1)-5=2k-5$ vertices by induction.
        Furthermore equality is possible for the colouring $c$ which assigns the colours $\{1,2\}$ to the end-vertices of $3$ edges of $M$ and assigns a unique colour for every other edge of $M$.
    \end{claimproof}
    
    \begin{claim}
        For $k\ge 3$, $\sn(G_k,k)=k.$
    \end{claim}
    \begin{claimproof}
        First, let $c$ be the colouring that assigns a unique colour for every edge $e$ in $M$ to the two end vertices of $e.$
        Now a partition class is a critical set for this colouring, these contain $k$ colours.
        Analogously to the proof of Claim~\ref{clm:sn(G,k-1)}, i.e. with some case analysis, one can prove that $\sn(G_k,k)\ge k$ and thus $\sn(G_k,k)= k.$
    \end{claimproof}
    By combining these two claims, we conclude that $\sn(G,k)=k<2k-5= \sn(G,k-1)$ if $k \ge 6.$ 
\end{proof}

Also $\underline{\lcs}(G,k)$ is not a non-decreasing parameter.

\begin{thr}
    There are graphs $G$ and integers $k$ for which $\underline{\lcs}(G,k)< \underline{\lcs}(G,k-1),$ i.e. the parameter $\underline{\lcs}$ is not monotone non-decreasing in the number of colours $k.$
\end{thr}

\begin{proof}

        Let $G$ be the graph formed by adding two $K_4^-$s ($K_4$ minus an edge) to a $K_{3,t}$ as presented in Figure~\ref{fig:ulcs(k-1)>ulcs(k)}.
    We will prove that if $t \ge 4$, then $\underline{\lcs}(G,4)<\underline{\lcs}(G,3).$
    Let $c$ be an arbitrary $3$-colouring of $G$.
    Since $u_1wx$ and $u_2wx$ are triangles, and so are $u_2yz$ and $u_3yz$, we note that $c(u_1)=c(u_2)=c(u_3)$.
    This implies that $c(v_i),$ for any $1 \le i \le t$ can never be a fixed vertex and as such the $v_i$ belong to every determining set.
    Since the colours could be swapped among $\{w,x\}$ and $\{y,z\},$ at least one of the two has to belong to a determining set.
    Hence $\underline{\lcs}(G,3) \ge t+2.$
    Actually equality does hold, since a larger determining set will not be minimal.
    
    Next, we let $c$ be the $4$-colouring with
    $c(v_i)=4$ for every $1 \le i \le t$, $c(u_i)=i$ for every $1 \le i \le 3$,
    and $c(w)=3,c(x)=4=c(y)$ and $c(z)=1.$
    One can check that $\lcs(G,c)=5$ and thus $\underline{\lcs}(G,4)\le 5.$
    To do this, note that at most one of the $v_i$s does belong to a critical set and so the case analysis is limited.
    Also $u_1$ and $u_3$ have to be in a determining set.
    Examples of critical sets are $\{u_1,u_2,u_3,x,y\}$ and $\{u_1,u_3,w,x,v_1\}.$
\end{proof}
\begin{figure}[h]
\begin{center}
    \begin{tikzpicture}
    {
    \foreach \x in {0,1,2,3,4,5,6}{
	\foreach \y in {0,3,6}{
	\draw[thick] (\y,0) -- (\x,2);}		
	;}

    \foreach \z in {1,2}{
	\foreach \y in {0,3}{
	\draw[thick] (\y,0) -- (\z,-1);}		
	;}
	
	\foreach \z in {4,5}{
	\foreach \y in {6,3}{
	\draw[thick] (\y,0) -- (\z,-1);}		
	;}
	
	\draw[thick] (2,-1) -- (1,-1);
	\draw[thick] (5,-1) -- (4,-1);

    \foreach \x in {0,1,2,3,4,5,6}{\draw[fill] (\x,2) circle (0.1);}
	\foreach \x in {0,3,6}{\draw[fill] (\x,0) circle (0.1);}
	\foreach \x in {1,2,4,5}{\draw[fill] (\x,-1) circle (0.1);}

    \coordinate [label=center: $v_1$] (A) at (0,2.35);
    \coordinate [label=center: $v_2$] (A) at (1,2.35);
    
    \coordinate [label=center: $\dots$] (A) at (3.5,2.35);
    \coordinate [label=center: $v_t$] (A) at (6,2.35);
    
	\coordinate [label=center: $u_1$] (A) at (0,-0.35);
    \coordinate [label=center: $u_2$] (A) at (3,-0.35);
    \coordinate [label=center: $u_3$] (A) at (6,-0.35);
    
    \coordinate [label=center: $w$] (A) at (1,-1.3);
    \coordinate [label=center: $x$] (A) at (2,-1.3);
    
    \coordinate [label=center: $y$] (A) at (4,-1.3);
    \coordinate [label=center: $z$] (A) at (5,-1.3);

	}
	\end{tikzpicture}
\end{center}
\caption{A graph $G$ for which $\underline{\lcs}(G,4)<\underline{\lcs}(G,3)$}\label{fig:ulcs(k-1)>ulcs(k)}
\end{figure}
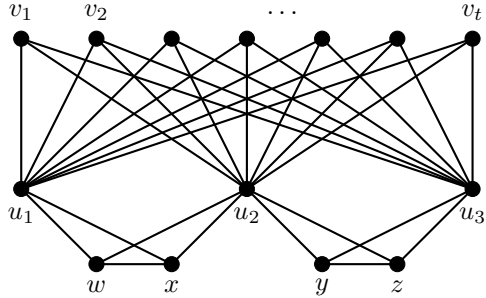
Finally, we observe that by Theorem~\ref{thr:overline=n}, the other two parameters 
are monotone non-decreasing in $k$.
\begin{cor}
    The parameters $\overline{\scs}(G,k)$ and $\overline{\lcs}(G,k)$ are monotone non-decreasing in $k$.
\end{cor}

\section{Monotone non-decreasing parameters for the subgraph-relation}\label{sec:monotonedecreasing_subgraph}

Since for every parameter $\rho \in \{ \sn, \underline{\lcs},\overline{\scs}, \overline{\lcs}\},$ we have $\rho(K_n)=n-1$ and $\rho(H)=1$ for a biparite graph, it might be tempting to wonder if the the parameters may be non-decreasing for the subgraph-relation.
Nevertheless, that is not the case, as shown by the followoing proposition.

\begin{prop}
For every parameter $\rho \in \{ \sn, \underline{\lcs},\overline{\scs}, \overline{\lcs}\},$
and any value $N>0$, there are examples of graphs $H \subsetneqq G$ with $\chi(H)=\chi(G)$ satisfying $\rho(H)>N\rho(G).$
\end{prop}

\begin{proof}
    It is easy to show (\cite[Thr.~6]{CK14}) that for any odd $p$, $\rho(C_{3p}) \ge \sn(C_{3p})=\frac{3p+1}{2}$, while (e.g. by \cite[Prop.~1]{CK14}) $\rho(K_{p,p,p})=2$ for each of the parameters.
    Since $C_{3p}$ is a subgraph of $K_{p,p,p}$, we conclude by taking $p\ge \frac{4N}{3}$.
\end{proof}

Note that we have added the condition that $\chi(H)=\chi(G)$ in this proposition.
when $H \subset G$ and $\chi(H) < \chi(G)$ one has more possible colours to extend the partial colouring in $G$ and so it is less natural to compare them.
In particular, we have that the parameter $\overline{\scs}$ is actually a non-increasing parameter when considering the subgraph-relation for graphs with the same chromatic number.

\begin{thr}
    Let $H \subset G$ with $\chi(H)=\chi(G)$.
    Then $\overline{\scs}(H) \ge \overline{\scs}(G).$
\end{thr}

\begin{proof}
    Assume that $\overline{\scs}(G)=\scs(G,c)$ for the $k$-colouring $c$, where $k=\chi(G).$
    Let $S$ be a smallest critical set for $(H,c).$
    Here we used that $c$ is also a proper colouring of $H$.
    Since $S$ is a determining set for $(H,c)$ and $G$ is a supergraph of $H$, $S$ will also be a determining set for $(G,c).$
    As such, we conclude that 
    $\overline{\scs}(G)=\scs(G,c) \le \lvert S \rvert = \scs(H,c) \le  \overline{\scs}(H).$
\end{proof}

We conjecture that the same is true for $\overline{\lcs}(H)$ and one can ask more specifically if for every proper colouring of $G$ under the conditions of the conjecture, $\lcs(H,c) \ge \lcs(G,c).$

\begin{conj}\label{conj:subgraphmonotonicity_for_olcs}
    Let $H \subset G$ with $\chi(H)=\chi(G)$.
    Then $\overline{\lcs}(H) \ge \overline{\lcs}(G).$
\end{conj}

The table in~~\cite[Fig.~1]{CK14}\footnote{the second graph in the second column has to be $K_5 \backslash P_3$} shows that there are no counterexamples with $5$ vertices for the following conjecture (with up to $5$ vertices actually).
By Theorem~\ref{thr:olcsG=n-1}, it is also true when $\overline{\lcs}(G)=n-1,$ as in that case $\overline{\lcs}(H)=n-1$ as well.

On the other hand, we remark that there is no monotonicity in terms of the subgraph-order for the Sudoku number $\sn(G)$ and $\underline{\lcs}(G).$
In both cases, the difference is unbouded.

\begin{prop}
    For any value of $N>1$, there are examples of graphs $H \subsetneqq G$ with $\chi(H)=\chi(G)$ satisfying $N \sn(H)<\sn(G).$
    The same is true for $\underline{\lcs}$ instead of $\sn$.
\end{prop}

\begin{proof}
    Let $G$ be a $K_{p,p}$ for which one of its vertices is connected to the two end-vertices of a $K_2$ (i.e. a $K_{p,p}$ and a $K_3$ that share exactly one vertex $v$).
    Let $H$ be equal to $G$ minus one edge of the $K_{p,p},$ not containing $v$ (otherwise $\sn(H)=4$).
    
    Smallest critical sets for some corresponding colourings $c$ have been represented in Figure~\ref{fig:snH<snG} for the case $p=5,$
    with the partial colouring in black and the remaining colours in red.
    Then one can check that $\sn(H)=3$ and $\underline{\lcs}(H) \le \lcs(H,c)=4.$
    For the verification, note that the vertices coloured $0$ in black belong to any critical set and at least one vertex of the triangle is needed (and at most $2$).
    On the other hand, $\sn(G)=\underline{\lcs}(G)=p+1.$
\end{proof}

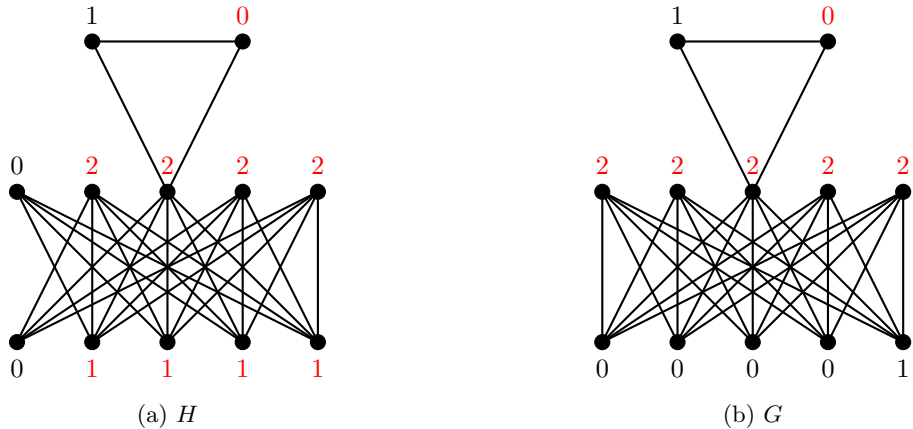
\begin{figure}[h]

\begin{minipage}[b]{.45\linewidth}
\begin{center}
    \begin{tikzpicture}
    {

    \foreach \x in {0,1,2,3,4}
    {
	\ifnum\x>0
	\foreach \y in {0,1,2,3,4}
	{
	\draw[thick] (\y,0) -- (\x,2);
	}		
	\draw[thick] (\x,0) -- (0,2);
	\fi
	;}

    \foreach \x in {0,1,2,3,4}{\draw[fill] (\x,2) circle (0.1);}
	\foreach \x in {0,1,2,3,4}{\draw[fill] (\x,0) circle (0.1);}
    \draw[fill] (1,4) circle (0.1);
	\draw[fill] (3,4) circle (0.1);	
	\draw[thick] (1,4)--(3,4)--(2,2)--cycle;
	\coordinate [label=center: $1$] (A) at (1,4.35);
    \coordinate [label=center: $0$] (A) at (0,2.35);
    \coordinate [label=center: $0$] (A) at (0,-0.35);
    \coordinate [label=center: \textcolor{red}{0}] (A) at (3,4.35);
    
    \foreach \x in {1,2,3,4}{\coordinate [label=center: \textcolor{red}{2}] (A) at (\x,2.35);
    \coordinate [label=center: \textcolor{red}{$1$}] (A) at (\x,-0.35);}
    }
	\end{tikzpicture}
\subcaption{$H$}
\label{fig:H1}
\end{center}
\end{minipage}\quad\begin{minipage}[b]{.45\linewidth}

\begin{center}
    \begin{tikzpicture}
    {
	\foreach \x in {0,1,2,3,4}{\foreach \y in {0,1,2,3,4}{
	\draw[thick] (\y,0) -- (\x,2);}				
	;}
	
    \foreach \x in {0,1,2,3,4}{\draw[fill] (\x,2) circle (0.1);}
	\foreach \x in {0,1,2,3,4}{\draw[fill] (\x,0) circle (0.1);}
    \draw[fill] (1,4) circle (0.1);
	\draw[fill] (3,4) circle (0.1);	
	\draw[thick] (1,4)--(3,4)--(2,2)--cycle;
	
	\coordinate [label=center: $1$] (A) at (1,4.35);
	\foreach \x in {0,1,2,3}{
	\coordinate [label=center: $0$] (A) at (\x,-0.35);
	}
	\coordinate [label=center: $1$] (A) at (
	4,-0.35);
	
	 \coordinate [label=center: \textcolor{red}{0}] (A) at (3,4.35);
    
    \foreach \x in {0,1,2,3,4}{\coordinate [label=center: \textcolor{red}{2}] (A) at (\x,2.35);
    }
	}
	\end{tikzpicture}\\
\subcaption{$G$}
\end{center}
\end{minipage}
\caption{Example of $H \subset G$ with $\sn(H)<\sn(G)$}\label{fig:snH<snG}
\end{figure}

\section{Hypergraph-versions for Sudoku and Latin squares}\label{sec:hypergraph}

The definition of the Sudoku number $\sn$, and also the three other variants, can be extended to hypergraphs, as has been done for Steiner systems in~\cite{BM21} for $\sn$.
A critical set for a pair $(H,c)$ of a hypergraph $H$ and colouring is again a minimal determining set, defined completely analogously. Here we focus on (weak) colourings of the hypergraph, i.e. no edge may be monochromatic.
For more info, see~\cite{BTV15}.
Natural examples are again e.g. the Sudoku and Latin square.
The Latin square hypergraph $K_n \square K_n$ can be defined as the hypergraphgraph on the $n^2$ vertices in a $n \times n$ grid, taking the vertices in the same row, or same column, as a hyperedge.
The hypergraph $\SUD_m$ is defined on $m^4$ vertices and $3m^2$ hyperedges (the initial $K_{m^2}$s) and is an extension of $K_{m^2} \square K_{m^2}.$
Here, we compute the $4$ parameters for these two examples.

\begin{exam}
When $m \ge 2$,
    $\overline{\scs}(\SUD_m)=\overline{\lcs}(\SUD_m)=m^4.$
    For $n \ge 4$, $\overline{\scs}(K_n \square K_n)=\overline{\lcs}(K_n \square K_n)=n^2.$
    For $n=3,$ $\overline{\scs}(K_n \square K_n)=\overline{\lcs}(K_n \square K_n)=6.$
\end{exam}

When the square has sidelengths $n, m^2$ at least $4$, one can take the chessboard-colouring that assigns alternately $0$ and $1$, to observe that $\overline{\scs}$ and $\overline{\lcs}$ are both equal to the order for both $K_n \square K_n$ and $\SUD_m.$
Those two parameters do not have to be equal to the order in general, e.g. $K_4^{(3)}$ (the $4$-clique with $4$ $3$-edges) has a essentially unique $2$-colouring and critical sets with $2$ vertices, and when $n=3$, the largest critical sets contain $6$ vertices, since in every row one can find at least one fixed vertex, and taking a diagonal with one colour and the remaining ones with the other colour, we see that $6$ can also be attained.

\begin{exam}
    For the Sudoku hypergraph $\SUD_2$, we have $\sn(\SUD_2)=\underline{\lcs}(\SUD_2)=9$.
\end{exam}

It has been verified by computer that $\sn(\SUD_2) \ge 9$ and 
the example at the left in Figure~\ref{fig:hypergraphSUD2coloring} gives an example for which the critical set (black numbers) is unique, being the complement of all fixed vertices (red numbers).

\begin{figure}[h]
\begin{minipage}[b]{.45\linewidth}
\begin{center}
\begin{tikzpicture}

  \begin{scope}
    \draw (0, 0) grid (4,4);
    \draw[very thick, scale=2] (0, 0) grid (2,2);

    \setcounter{row}{1}
    \setrow {0}{1 }  {0}{\textcolor{red}{1} }
    \setrow {0}{\textcolor{red}{1} }  {0}{ 0}

  \setrow {0}{ \textcolor{red}{0}}  {0}{\textcolor{red}{1} }
    \setrow {\textcolor{red}{1}}{ 1}  {\textcolor{red}{1}}{\textcolor{red}{0} }

  \end{scope}
  	\end{tikzpicture}
\end{center}

\end{minipage}\quad\begin{minipage}[b]{.45\linewidth}
\begin{center}
\begin{tikzpicture}

  \begin{scope}
    \draw (0, 0) grid (4,4);
    \draw[very thick, scale=2] (0, 0) grid (2,2);

    \setcounter{row}{1}
    \setrow {0}{1 }  {1}{1 }
    \setrow {0}{1 }  {0}{ 0}

  \setrow {0}{0}{1}{0}
    \setrow {1}{1 }  {1}{ 0}

  \end{scope}
\end{tikzpicture}\\
\end{center}
\end{minipage}
    \caption{Hypergraph $2$-colouring of $\SUD_2$ and $K_4 \square K_4$}
    \label{fig:hypergraphSUD2coloring}
\end{figure}

\begin{prop}
    For $m\ge 3$, $\sn(\SUD_m)=\underline{\lcs}(\SUD_m)=m^4-3m^2+3m$.
\end{prop}

\begin{proof}
    Note that the vertices of the hypergraph $\SUD_m$ are the elements in the $m^2 \times m^2$ square and we will use element, cell or vertex for the same concept.
    
    Since the complement of a determining set only contains fixed vertices, it is sufficient to observe that the number of fixed vertices is bounded by $3m(m-1)$ and there is a colouring with a critical set of the desired form.
    Figure~\ref{fig:hypergraphSUD3coloring} shows such a colouring for $m=3$, but the same one extends to every value $m\ge 3$. 
    
    Call a $m\times m$ square fixed if all of its elements except from one, $v$, have the same colour. Analogously, we call a row or column fixed.
    The colouring in Figure~\ref{fig:hypergraphSUD3coloring} has $m(m-1)$ fixed $m \times m$ squares, columns and rows for which all elements different from the fixed one also do belong to the critical set.

    Now, we prove that there do not exist colourings with more than $3m(m-1)$ fixed vertices.
    In the row (or column) of the fixed vertex $v$ of a fixed $m \times m$ square, there are already $m-1 \ge 2$ elements coloured with the other colour, so if the row or column would be fixed, $v$ would be the associated fixed vertex.
    
    Consider the $m \times m$ squares within the $m^2 \times m^2$ square.
    Every $m$ horizontal $m \times m$ squares form a horizontal line (with $m^3$ cells) and similarly there are $m$ vertical lines.
    
    First assume there is at least a vertical line and a horizontal line that each contain $m$ fixed $m\times m$ squares.
    But then the $m$ fixed vertices do have to belong to $m$ different rows resp. columns and so there are at least $2m-1$ columns and rows that do not contain a vertex that is fixed for that row or column (different from the one fixed for its $m\times m$ square).
    So in this case, there would be at most $3m^2-4m+2<3m(m-1)$ fixed vertices.
    As the latter is not the case, there are no more than $m(m-1)$ fixed $m\times m$ squares.
    If there is at least one vertical line (or horizontal one) containing only fixed $m \times m$ squares, we can conclude immediately as there are at most $m(m-1)$ fixed rows and columns as well (and equality cannot occur for the $3$ types). 
    Now, as long as there are more than $m(m-2)$ fixed $m \times m$ squares, there are also $m$ rows and columns that are not fixed.
    If there are more than $(m-1)(m-2),$ there are at least $m-1$ rows (and columns as well) that are not fixed. So we can conclude in all cases.
\end{proof}

\begin{figure}
 \centering
\begin{tikzpicture}

    \begin{scope}
    \draw (0, 0) grid (9, 9);
    \draw[very thick, scale=3] (0, 0) grid (3, 3);
    
    \setcounter{row}{1}
    \setrowup {0}{0 }{0}  {0 }{0}{0 }  {0}{0}{}
    \setrowup {0}{0 }{0}  {0 }{0}{0 }  {0}{0}{}
    \setrowup {0}{0 }{}  {0 }{0}{ }  {}{}{0}
    
    \setrowup {0}{0 }{0}  {0 }{0}{ }  {0}{0}{0}
    \setrowup {0}{0 }{0}  {0 }{0}{ }  {0}{0}{0}
    \setrowup {0}{0 }{}  {}{}{0} {0 }{0}{ }  
    
    \setrowup {0}{0 }{}  {0 }{0}{0 }  {0}{0}{0}
    \setrowup {0}{0 }{}  {0 }{0}{0 }  {0}{0}{0}
    \setrowup  {}{}{0}    {0}{0 }{}  {0 }{0}{ } 
    
    \begin{scope}[red, font=\sffamily\slshape]
      \setcounter{row}{1}
      \setrowup {}{ }{}  { }{}{ }  {}{ }{1}
      \setrowup {}{ }{}  { }{}{ }  {}{ }{1}
      \setrowup {}{ }{1}  { }{}{1 }  {1}{1 }{}
      
    \setrowup {}{ }{}  { }{}{ 1}  {}{ }{}
      \setrowup {}{ }{}  { }{}{1 }  {}{ }{}
      \setrowup {}{ }{1}  {1 }{1}{ }  {}{ }{1}
      
    \setrowup {}{ }{1}  { }{}{ }  {}{ }{}
      \setrowup {}{ }{1}  { }{}{ }  {}{ }{}
      \setrowup {1}{ 1}{}  { }{}{1 }  {}{ }{1}
    \end{scope}

\end{scope}
\end{tikzpicture}
   \caption{Hypergraph $2$-colouring of $\SUD_3$}
    \label{fig:hypergraphSUD3coloring}
\end{figure}

\begin{exam}
For $n\ge 4$, 
$\sn(K_n \square K_n)=\underline{\lcs}(K_n \square K_n)=(n-1)^2.$
For $n \in \{2,3\}$, these parameters equal $1$ and $4$ respectively.
\end{exam}

For $n \ge 4$, colour a $(n-1)\times (n-1)$-square of the Latin square with one colour, then the remaining $2n-1$ vertices are fixed. As such there is a unique critical set for this colouring.
Furthermore, a vertex can only be fixed if it is the unique vertex coloured with one colour in a hyperedge. Hence there are at most $2n$ fixed vertices, and these should not be fixed for both its row and column. 
The latter is impossible, except for $n=4.$
Without loss of generality, $(1,1)$ is coloured $0$ while the remaining of its row is completely coloured $1$ and its column is coloured $0$ except from $(n,1).$
The row of $(n,1)$ is coloured fully $1,$ except from one element coloured zero, without loss of generality $(n,n)$. 
Then the $n^{th}$ column has to be coloured fully zero, except from $(1,n).$
Rows $2$ till $n-1$ have two vertices coloured in $0$ and for every such row, all vertices except from one are coloured $0.$
For columns $2$ till $n-1$, all the vertices in the column are coloured $1$, except from one vertex.
If $n-2>2$, this gives a contradiction.
For $n=4$, there are colourings with $8$ fixed vertices, but the complement of the fixed vertices is not a determining set (essentially unique; see right square in Figure~\ref{fig:hypergraphSUD2coloring}).
For $n=3,$ one can take the right lower $3 \times 3$ square in the left figure of Figure~\ref{fig:hypergraphSUD2coloring} as an example.

\section{Final remarks}

In this paper, we investigated the behaviour of critical sets when the number of colours is allowed to be larger than the chromatic number of the graph.
We proved two basic characterizations for the parameters to be equal to the maximum, being $n-1$ respectively $n$. 
Here the case $\underline{\lcs}(G)=n-1$ seemed to be the exception where the structure is not as clear as in the other cases.

We proved that $\overline{\scs}$ and $\overline{\lcs}$ are both monotone non-decreasing in the number of colours, while that is not the case for $sn=\underline{\scs}$ and $\underline{\lcs}$.
This solved problem $3$ (mentioned in the conclusion section) of~\cite{CK14}.
We observed the same behaviour when looking to the subgraph relation for graphs with the same chromatic number, with the case $\overline{\lcs}$ being a conjecture.

Finally, we considered the same problem for hypergraphs with weak colourings and determined the Sudoku numbers for the Sudoku - and Latin square hypergraphs.

There are still some interesting open questions in the area.

The study of critical sets of Latin squares is attributed to J. Nelder in $1977$.
The conjecture that $\sn(K_n \square K_n) = \floorfrac{n^2}{4}$, as mentioned in e.g.~\cite{BvR99} is probably still one of the most intriguing ones.
Also the other $3$ parameters for Latin squares are unknown, see e.g.~\cite{HM03} where they prove that $\underline{\lcs}(K_n \square K_n) =n^2(1-o(1)).$
Determining $\sn(K_n \square K_n,k)$ or $\sn(\SUD_m,k)$ for $k>n$ or $k>m^2$ is a related question where the patterns might be equally interesting.

Having solved one of the $5$ problems in the conclusion of Cooper et al.~\cite{CK14}, there are still some remaining.
One of them being the question if $G$ has to be uniquely colourable if every critical set of $G$ has cardinality $\chi(G).$
We verified that it is true if $G$ can be coloured in essentially two ways (i.e. $G$ has $2\chi!$ possible colourings).
\begin{prop}
Assume $G$ can be coloured (up to isomorphism) in exactly two ways.
Then for any choice of $(G,c)$ there does exist a critical set with cardinality $\chi(G).$
\end{prop}

\begin{proof}
Let $k=\chi(G)$ and let the two partitions in independent sets be $V=V_1 \cup V_2 \cup \ldots \cup V_k$ and $V=U_1 \cup U_2 \cup \ldots \cup U_k.$
Since $k=\chi(G),$ the union of $j$ sets $V_i$ intersect at least $j$ sets $U_i$ and vice versa.
In particular, by Hall's matching theorem, there exists a common transversal of the two partitions. Without loss of generality we may assume there are elements $(x_1,x_2 \ldots x_k)$ with $x_i \in U_i \cap V_i$ for every $1 \le i \le k.$
Since the partitions are different (there are some $i$ for which $U_i \not =V_i$), we may also assume without loss of generality that $V_1 \cap U_2 \not=\emptyset.$
Choose a vertex $y_2 \in V_1 \cap U_2.$
Let $c$ be a colouring for which the colour classes are precisely the $V_i.$
Now $D=\{x_1,y_2,x_3,x_4 \ldots x_k\}$ will be a determining set for $(G,c)$, since knowing that $x_1$ and $y_2$ have the same colour implies that we know the partition and we know the colours for each partition class (except one, which needs the remaining colour).
Furthermore it is a critical set, since deleting $x_1$ or $y_2$ implies that we cannot distinguish the two partitions, and deleting $x_i$ for some $i \ge 3$ from this set implies that there are only $k-2$ colours in the restriction of $c$ to $D\backslash x_i.$
\end{proof}

There are also two conjectures posed in this work.
Conjecture~\ref{conj:subgraphmonotonicity_for_olcs} would be the final case related to monotonicity of the subgraph-order.
Also resolving Conjecture~\ref{conj:sn=n-2_no_Kn} could be nice, having only two exceptional cases.
Proving the weaker form that $\sn(G)=n-2$ if $\chi(G)=k$ is sufficiently large implies that $G$ contains $K_k$ as a subgraph, could be considered as analogous behaviour to the extension of Brook's theorem in~\cite{Reed99}.

\section*{Acknowledgement}
The author would like to thank John Haslegrave for making time for a chat about this project, as well as the group of Lau et al.

\paragraph{Open access statement.} For the purpose of open access,
a CC BY public copyright license is applied
to any Author Accepted Manuscript (AAM)
arising from this submission.

\bibliographystyle{abbrv}
\bibliography{sn}

\begin{thebibliography}{10}

\bibitem{AHM04}
P.~Afshani, H.~Hatami, and E.~S. Mahmoodian.
\newblock On the spectrum of the forced matching number of graphs.
\newblock {\em Australas. J. Combin.}, 30:147--160, 2004.

\bibitem{ABKM22}
S.~Akbari, A.~Beikmohammadi, S.~Klav\v{z}ar, and N.~Movarraei.
\newblock On the chromatic vertex stability number of graphs.
\newblock {\em European J. Combin.}, 102:Paper No. 103504, 7, 2022.

\bibitem{Albertson98}
M.~O. Albertson.
\newblock You can't paint yourself into a corner.
\newblock {\em J. Combin. Theory Ser. B}, 73(2):189--194, 1998.

\bibitem{BvR99}
J.~A. Bate and G.~H.~J. van Rees.
\newblock The size of the smallest strong critical set in a {L}atin square.
\newblock {\em Ars Combin.}, 53:73--83, 1999.

\bibitem{BM21}
N.~Besharati and M.~Mortezaeefar.
\newblock Determination of the size of defining set for steiner triple systems.
\newblock {\em Journal of Discrete Mathematical Sciences and Cryptography},
  0(0):1--19, 2021.

\bibitem{BTV15}
C.~Bujt\'{a}s, Z.~Tuza, and V.~Voloshin.
\newblock Hypergraph colouring.
\newblock In {\em Topics in chromatic graph theory}, volume 156 of {\em
  Encyclopedia Math. Appl.}, pages 230--254. Cambridge Univ. Press, Cambridge,
  2015.

\bibitem{CK14}
J.~Cooper and A.~Kirkpatrick.
\newblock Critical sets for {S}udoku and general graph colorings.
\newblock {\em Discrete Math.}, 315:112--119, 2014.

\bibitem{DMRP03}
D.~Donovan, E.~S. Mahmoodian, C.~Ramsay, and A.~P. Street.
\newblock Defining sets in combinatorics: a survey.
\newblock In {\em Surveys in combinatorics, 2003 ({B}angor)}, volume 307 of
  {\em London Math. Soc. Lecture Note Ser.}, pages 115--174. Cambridge Univ.
  Press, Cambridge, 2003.

\bibitem{FJ85}
J.~F. Fink and M.~S. Jacobson.
\newblock {$n$}-domination in graphs.
\newblock In {\em Graph theory with applications to algorithms and computer
  science ({K}alamazoo, {M}ich., 1984)}, Wiley-Intersci. Publ., pages 283--300.
  Wiley, New York, 1985.

\bibitem{HSV07}
F.~Harary, W.~Slany, and O.~Verbitsky.
\newblock On the computational complexity of the forcing chromatic number.
\newblock {\em SIAM J. Comput.}, 37(1):1--19, 2007.

\bibitem{HM03}
H.~Hatami and E.~S. Mahmoodian.
\newblock A lower bound for the size of the largest critical sets in {L}atin
  squares.
\newblock {\em Bull. Inst. Combin. Appl.}, 38:19--22, 2003.

\bibitem{HY18}
H.~Hatami and Y.~Qian.
\newblock Teaching dimension, {VC} dimension, and critical sets in {L}atin
  squares.
\newblock {\em J. Comb.}, 9(1):9--20, 2018.

\bibitem{Lau22}
G.-C. {Lau}, J.~M. {Jeyaseeli}, W.-C. {Shiu}, and S.~{Arumugam}.
\newblock {Sudoku Number of Graphs}.
\newblock {\em arXiv e-prints}, page arXiv:2206.08106, June 2022.

\bibitem{M98}
E.~S. Mahmoodian.
\newblock Defining sets and uniqueness in graph colorings: a survey.
\newblock volume~73, pages 85--89. 1998.
\newblock R. C. Bose Memorial Conference (Fort Collins, CO, 1995).

\bibitem{MTC14}
G.~McGuire, B.~Tugemann, and G.~Civario.
\newblock There is no 16-clue {S}udoku: solving the {S}udoku minimum number of
  clues problem via hitting set enumeration.
\newblock {\em Exp. Math.}, 23(2):190--217, 2014.

\bibitem{D05}
D.~Mojdeh.
\newblock On the defining spectrum of {$k$}-regular graphs with {$k-1$} colors.
\newblock {\em J. Prime Res. Math.}, 1(1):118--135, 2005.

\bibitem{P22}
A.~{Pokrovskiy}.
\newblock {Graphs with Sudoku number $n-1$}.
\newblock {\em arXiv e-prints}, page arXiv:2206.08914, June 2022.

\bibitem{Reed99}
B.~Reed.
\newblock A strengthening of {B}rooks' theorem.
\newblock {\em J. Combin. Theory Ser. B}, 76(2):136--149, 1999.

\bibitem{Zsolt97survey}
Z.~Tuza.
\newblock Graph colorings with local constraints---a survey.
\newblock {\em Discuss. Math. Graph Theory}, 17(2):161--228, 1997.

\end{thebibliography}

\end{document}